\documentclass[12pt]{article}
\usepackage{a4}
\usepackage{amsthm}
\usepackage{amsmath}
\usepackage{amssymb}
\usepackage{amsfonts}
\usepackage{cite}
\newtheorem{theorem}{Theorem}[section]
\newtheorem{lemma}[theorem]{Lemma}

\newtheorem{definition}[theorem]{Definition}

\begin{document}

\title{More non-bipartite forcing pairs}

\author{Tam\'as Hubai\thanks{Institute of Mathematics, E\"otv\"os Lor\'and University, Budapest.  Previous affiliation: Department of Computer Science and DIMAP, University of Warwick, Coventry CV4 7AL, UK. E-mail: {\tt htamas@cs.elte.hu}. This author was supported by the Engineering and Physical Sciences Research Council Standard Grant number EP/M025365/1.}\and
        Dan Kr\'al\thanks{Faculty of Informatics, Masaryk University, Botanick\'a 68A, 602 00 Brno, Czech Republic, and Mathematics Institute, DIMAP and Department of Computer Science, University of Warwick, Coventry CV4 7AL, UK. E-mail: {\tt dkral@fi.muni.cz}. This author was supported by the European Research Council (ERC) under the European Union's Horizon 2020 research and innovation programme (grant agreement No 648509). This publication reflects only its authors' view; the European Research Council Executive Agency is not responsible for any use that may be made of the information it contains.}\and
        Olaf Parczyk\thanks{Institut f\"ur Mathematik, Technische Universit\"at Ilmenau, 98684 Ilmenau, Germany. E-mail: {\tt olaf.parczyk@tu-ilmenau.de} and {\tt yury.person@tu-ilmenau.de}. These authors were supported by DFG grant PE 2299/1-1.}\and
\newcounter{lth}
\setcounter{lth}{3}
        Yury Person$^\fnsymbol{lth}$}

\date{}

\maketitle

\begin{abstract}
We study pairs of graphs $(H_1,H_2)$ such that
every graph with the densities of $H_1$ and $H_2$ close to the densities of $H_1$ and $H_2$ in a random graph is quasirandom;
such pairs $(H_1,H_2)$ are called forcing.
Non-bipartite forcing pairs were first discovered by Conlon, H\`an, Person and Schacht~[{\em Weak quasi-randomness
for uniform hypergraphs}, Random Structures Algorithms \textbf{40} (2012),
no.~1, 1--38]:
they showed that $(K_t, F)$ is forcing
where $F$ is the graph that arises from $K_t$ by iteratively doubling its vertices and edges in a prescribed way $t$ times.
Reiher and Schacht~[{\em Forcing quasirandomness with triangles}, 
Forum of Mathematics, Sigma. Vol. 7, 2019] strengthened this result for $t=3$ by proving that two doublings suffice and
asked for the minimum number of doublings needed for $t>3$.
We show that $\lceil(t+1)/2\rceil$ doublings always suffice.
\end{abstract}

\section{Introduction and results}
The systematic study of quasirandom graphs has been initiated by Thomason~\cite{Thomason-2,Thomason-1} and
Chung, Graham and Wilson~\cite{CGW89} in the 1980's.
Since then, many properties of quasirandom graphs were described.
We refer to the surveys~\cite{KSSSz02,KS06}. 

A key property of a quasirandom graph is an almost uniform edge distribution.
A sequence $(G_n)_{n\in \mathbb{N}}$ of graphs is {\em $p$-quasirandom}
if $e_{G_n}(U)=p\binom{|U|}{2}+o(|G_n|^2)$ for all subsets $U\subseteq V(G_n)$,
where $|G_n|$ is the number of vertices of $G_n$ and $e_{G_n}(U)$ is the number of edges of $G_n$ with both end vertices in $U$.
A particular graph $G$ with $n$ vertices is {\em $(\varepsilon,p)$-quasirandom} if
$\left|e_{G}(U)-p\binom{|U|}{2}\right|\le \varepsilon n^2$ for all subsets $U\subseteq V(G)$.

One of many equivalent characterizations of $p$-quasirandom sequences $(G_n)_{n\in \mathbb{N}}$ of graphs is the following:
$(G_n)_{n\in \mathbb{N}}$ is quasirandom if and only if 
$G_n$ has edge density $p$ and contains $(p^4+o(1))|G_n|^4$ labelled (non-induced) copies of $C_4$.
Equivalently, $G_n$ contains asymptotically the expected number of copies of $K_2$ and $C_4$ as the Erd\H os-R\'enyi random graph $G(n,p)$.
This leads to the definition of a forcing pair of graphs given below.
To give the definition, we need to introduce the following notation.
If $F$ and $G$ are two graphs,
then $t(F,G)$ is the number of graph homomorphisms from $F$ to $G$,
i.e.\ all maps $f : V(F) \rightarrow V(G)$ with $f(e) \in E(G)$ for all $e\in E(F)$.
In addition, we write $e(F)$ for the number of edges of $F$.

\begin{definition}[Forcing pairs]
	A pair $(F_1,F_2)$ is called {\em forcing}
	if for every  $p \in (0,1]$ and $\varepsilon>0$, there exists a $\delta>0$ such that the following holds.
	Every graph $G$ with
	\begin{align*}
	t(F_1,G) = (1\pm\delta) p^{e(F_1)} \qquad \text{and} \qquad t(F_2,G) = (1\pm\delta) p^{e(F_2)}
	\end{align*}
	is $(\varepsilon,p)$-quasirandom.
\end{definition}

In particular, the pair $(K_2,C_4)$ is forcing. There are two exciting conjectures related to forcing pairs: Sidorenko's conjecture made independently by Sidorenko~\cite{Sid93} and 
by Erd\H{o}s and Simonovits~\cite{Sim84}, and the so-called forcing conjecture made by Skokan and Thoma~\cite{ST04}. While Sidorenko's conjecture asks whether the lower bound on $t(F_2,G)$ is always at least $p^{e(F_2)}$ where $p=t(K_2,G)$, the forcing conjecture states that  any pair $(K_2,F)$, where $F$ is a bipartite graph containing a cycle, is forcing. Due to their relation to Szemer\'edi's regularity lemma, these conjectures expedited tremendous amount of research in extremal combinatorics. 
Thus, additional forcing pairs were studied in~\cite{CGW89,CFS,conlon2012weak,han2011note,HH09,reiher2017forcing,ST04}. 
The first non-bipartite forcing pairs were found by Conlon et al.~in~\cite{conlon2012weak}.
So far, all non-bipartite forcing pairs are obtained by the construction described below.

Let $F$ be a $t$-partite graph and $V(F)=V_1(F) \dot\cup \ldots \dot\cup V_t(F)$ a fixed $t$-coloring of $F$.
The doubling $\mathcal{T}(F)$ on $V_1(F)$ is the graph obtained by taking two identical disjoint copies $F_1$ and $F_2$ of $F$ and
identifying the corresponding vertices in $V_1(F_1)$ and $V_1(F_2)$.
In this way,
we obtain a $t$-coloring of $\mathcal{T}(F)$ given by the sets $V_1(F_1)=V_1(F_2)$, $V_2(F_1) \dot\cup V_2(F_2)$, \ldots, $V_t(F_1) \dot\cup V_t(F_2)$.
For $k\le t$,
the $k$-fold doubling $\mathcal{T}_k(F)$ is defined as the doubling $\mathcal{T}(F)$ on $V_1(F)$ for $k=1$ and
the doubling $\mathcal{T}(\mathcal{T}_{k-1}(F))$ on $V_k(\mathcal{T}_{k-1}(F))$ for $k \ge 2$.
The order of the doublings has no influence on $\mathcal{T}_k(F)$, i.e.\ we could permute $V_1(F),\dots,V_k(F)$ arbitrarily.
Observe that $\mathcal{T}_2(K_2) = C_4$.

The pair $(K_2,C_4)=(K_2,\mathcal{T}_2(K_2))$ is forcing.
The result from~\cite{conlon2012weak} states that
the pair $(K_t,\mathcal{T}_t(K_t))$ is also forcing for any $t\ge 3$.
H\`an et al.~\cite{han2011note} generalized this result for any $t$-colorable graph $F$ in a similar way.
Reiher and Schacht~\cite{reiher2017forcing} improved the result from~\cite{conlon2012weak} for $t=3$ by showing that
the pair $(K_3,\mathcal{T}_2(K_3))$ is forcing.
We generalize this result for $t>3$;
we note that the same was independently proven by Reiher and Schacht~\cite{MS_comm}.

\begin{theorem}
	\label{thm:main_qrg}
	The pair $(K_t,\mathcal{T}_{\lceil (t+1)/2 \rceil }(K_t))$ is forcing for any $t\ge 2$.
\end{theorem}

We will present the proof in the language of graph limits, which we now introduce,
since this makes the arguments particularly short and transparent. 
Let $W$ be a {\em kernel}, i.e.\ a bounded symmetric Lebesgue measurable function from $[0,1]^2$.
We write $W \equiv p$ if $W$ is equal to $p$ almost everywhere;
a kernel with $W \equiv p$ is called {\em $p$-quasirandom}.
The homomorphism density extends in a natural way for a graph $F$ and a kernel $W$:
\begin{align*}
t(F,W)=\int_{[0,1]^{V(F)}} \prod_{uv \in E(F)} W(x_u,x_v) \prod_{u \in V(F)} \operatorname{d}x_u.
\end{align*}
A {\em graphon} is a kernel $W$ with values in $[0,1]$.
A pair of graphs $(F_1,F_2)$ is called {\em forcing}
if for every real $p \in (0,1]$,
every graphon $W$ with $t(F_1,W)=p^{e(F_1)}$ and $t(F_2,W)=p^{e(F_2)}$ is $p$-quasirandom.
This definition, see~\cite[Chapter 16]{lovasz2012large}, coincides with the definition of a forcing pair given earlier.

\section{Proof of Theorem~\ref{thm:main_qrg} for $t=4$}

In this section, we give a proof of Theorem~\ref{thm:main_qrg} for $K_4$.
We need the following lemma.

\begin{lemma}[Lemma~10 from~\cite{conlon2012weak}]\label{lem:CHPS_CS}
	Let $W$ be a graphon and $p \in (0,1]$ such that
	\begin{align*}
	t(K_4,W) = p^6 \qquad \text{and} \qquad t(\mathcal{T}_3(K_4),W) = p^{48}.
	\end{align*}
	Then
	\begin{align*}
	W(x_1,x_2) W(x_1,x_3) W(x_2,x_3)   \int_{[0,1]} W(x_1,y) W(x_2,y) W(x_3,y) \operatorname{d}y  = p^6
	\end{align*}
	for almost all $(x_1,x_2,x_3) \in [0,1]^3$.\qed
\end{lemma}
The proof of this lemma is given in~\cite{conlon2012weak} in the language of quasi\-random (hyper)graphs, and
we sketch the line of arguments here for completeness.
It can be proved by repeatedly applying Cauchy-Schwarz inequality starting to $t(K_4,W)$ three times.
This series of applications of Cauchy-Schwarz inequality yields that $t(\mathcal{T}_3(K_4),W) \ge p^{48}$.
Since it holds $t(\mathcal{T}_3(K_4),W) = p^{48}$ by the assumption of the lemma,
it follows that for almost all values of  $(x_1,x_2,x_3) \in [0,1]^3$ one has 
\[
W(x_1,x_2) W(x_1,x_3) W(x_2,x_3)   \int_{[0,1]} W(x_1,y) W(x_2,y) W(x_3,y) \operatorname{d}y  = p^6.
\] 

The next lemma together with Lemma~\ref{lem:CHPS_CS} readily implies Theorem~\ref{thm:main_qrg} for $k=4$.
\begin{lemma}\label{lem:RS_general}
	Let $W$ be a graphon and $p \in (0,1]$. If it holds that
	\begin{align}
	\label{eq:graphon-matrix}
	W(x_1,x_2) W(x_1,x_3) W(x_2,x_3)   \int_{[0,1]} W(x_1,y) W(x_2,y) W(x_3,y) \operatorname{d}y  = p^6
	\end{align}
	for almost all $(x_1,x_2,x_3) \in [0,1]^3$,
	then $W$ is $p$-quasirandom.
\end{lemma}

Before presenting the proof,
we recall the definition of {\em the essential supremum} $\operatorname{ess} \sup(f)$ of a (Lebesgue) measurable function $f \colon \mathbb{R}^n \to \mathbb{R}$.
It is the infimum over all $y \in \mathbb{R}$ with $f(x) \le y$ for almost all $x \in \mathbb{R}^n$,
i.e.\ $\operatorname{ess} \sup(f):=\inf\left\{y\colon \lambda(\{f\ge y\})=0\right\}$, where $\lambda$ is the Lebesgue measure.
The {\em essential infimum} of a function is defined analogously. 

\begin{proof}[Proof of Lemma~\ref{lem:RS_general}]
	Let $f : [0,1] \rightarrow [0,1]$ be defined as
	\[
	f(x)=\sup_{y_1,y_2\in[0,1]}(W(x,y_1)-W(x,y_2)).
	\]
	Observe that $f$ is a measurable function and set $c:=\operatorname{ess}\sup f\in [0,1]$.
	If $c=0$, then $W$ is $p$-quasirandom. Thus, we assume that $c>0$.
	
	The definition of $f$ and $c$ implies that there exist reals $a,b \in \mathbb{R}$ with $c=b-a$ satisfying the following.
	For any $\eta > 0$, there exist $x_1,x_2,x_3,x_4 \in [0,1]$ such that
	\begin{equation}
	\begin{aligned}
	\label{eq:supbound}
	W(x_1,x_2) &\ge b - \eta, \qquad W(x_1,x_3) &\ge b - \eta, \\
	W(x_2,x_4) &\le a + \eta, \qquad W(x_3,x_4) &\le a + \eta,
	\end{aligned}
	\end{equation}
	and~\eqref{eq:graphon-matrix} holds for $(x_1,x_2,x_3)$ and $(x_2,x_3,x_4)$.
	In addition,
	we can assume that $a>0$, $W(x_2,x_3)>0$, $\operatorname{ess}\sup_{y\in [0,1]} W(x_2,y) W(x_3,y)>0$ (because $p>0$), and that
	\begin{align}
	\label{eq:genbound}
	W(x_1,y) \ge a -\eta \quad \text{and} \quad W(x_4,y) \le b + \eta \quad \text{for almost all } y\in [0,1].
	\end{align}
	
	We get from \eqref{eq:graphon-matrix} that
	\begin{align*}
	W(x_1,x_2) W(x_1,x_3) W(x_2,x_3)   \int_{[0,1]} W(x_1,y) W(x_2,y) W(x_3,y) \operatorname{d}y   \\
	= p^6 = W(x_2,x_4) W(x_3,x_4) W(x_2,x_3)   \int_{[0,1]} W(x_4,y) W(x_2,y) W(x_3,y) \operatorname{d}y 
	\end{align*}
	Using \eqref{eq:supbound} and \eqref{eq:genbound} we can lower bound the left hand side by
	\begin{align*}
	(b-\eta)^2 (a-\eta) W(x_2,x_3)   \int_{[0,1]} W(x_2,y) W(x_3,y) \operatorname{d}y
	\end{align*}
	and similarly we can upper bound the right hand side by 
	\begin{align*}
	(a+\eta)^2 (b+\eta) W(x_2,x_3)   \int_{[0,1]} W(x_2,y) W(x_3,y) \operatorname{d}y.
	\end{align*}
	
	As $W(x_2,x_3) \int_{[0,1]} W(x_2,y) W(x_3,y) \operatorname{d}y$ is non-zero, we obtain that
	\begin{align*}
	(b-\eta)^2 (a-\eta) \le (a+\eta)^2 (b+\eta)
	\end{align*}
	and then deduce that $b \le a + 10\eta (b/a)$. 
	Together with the assumption that $b > a$ this implies that $W(x,y) \in [p \pm 20\eta (b/a)]$ for almost all $(x,y) \in [0,1]^2$.
	Since this holds for every $\eta>0$, the lemma follows.
\end{proof}

\section{Proof of Theorem \ref{thm:main_qrg}---general case}

The proof of the general case is based on the same idea as used in the previous section and
follows from the next two lemmas.
The first lemma can be proven by repeated applications of the Cauchy-Schwarz inequality similarly to the proof of Lemma~\ref{lem:CHPS_CS}.

\begin{lemma}\label{lem:CSI}
	Let $W$ be a graphon, $p \in (0,1]$, $t \ge 3$, and $k:=\lceil (t+1)/2 \rceil$ such that
	\begin{align*}
	t(K_t,W) = p^{e(K_t)} \qquad \text{and} \qquad t(\mathcal{T}_k(K_t),W) = p^{e(\mathcal{T}_k(K_t))}.
	\end{align*}
	Then
	\begin{equation*}
	\begin{aligned}
	p^{e(K_t)} = &\left( \prod_{1 \le i < j \le k} W(x_i,x_j) \right)  \\
		& \cdot  \int_{[0,1]^{t-k}} \left(  \prod_{1 \le i < j \le t-k} W(y_i,y_j) \right)  \left(   \prod_{i \in [k], j \in [t-k]} W(x_i,y_j) \right) \prod_{i \in [t-k]} \operatorname{d}y_i .
	\end{aligned}
	\end{equation*}
	for almost all $(x_1,\dots,x_k) \in [0,1]^k$.
\end{lemma}

\begin{lemma}\label{lem:RS_large_k}
	Let $W$ be a graphon, $p \in (0,1]$, $t\ge 3$, and $k:=\lceil (t+1)/2 \rceil$.
	If it holds that
	\begin{equation}
	\begin{aligned}
	\label{eq:graphon-matrix2}
	p^{e(K_t)}& =  \left( \prod_{1 \le i < j \le k} W(x_i,x_j) \right)  \\
	& \cdot \int_{[0,1]^{t-k}} \left(  \prod_{1 \le i < j \le t-k} W(y_i,y_j) \right) \left(   \prod_{i \in [k], j \in [t-k]} W(x_i,y_j) \right) \prod_{i \in [t-k]} \operatorname{d}y_i .
	\end{aligned}
\end{equation}
	for almost all $(x_1,\dots,x_k) \in [0,1]^k$,
	then $W$ is $p$-quasirandom.
\end{lemma}

\begin{proof}
	Let $f : [0,1] \rightarrow [0,1]$ be defined as
	\[
	f(x)=\sup_{y_1,y_2\in[0,1]} (W(x,y_1)-W(x,y_2)).
	\]
	Again, $f$ is a measurable function and we set $c:=\operatorname{ess} \sup f \in [0,1]$.
	If $c=0$, then $W$ is $p$-quasirandom.
	So, we assume that $c>0$ and consider positive reals $a,b \in \mathbb{R}$ with $c=b-a$ such that the following holds.
	For any $\eta > 0$, there exist $x_1,\dots,x_{k+1} \in [0,1]$ such that
	\begin{align}
	\label{eq:supbound2}
	W(x_1,x_i) \ge b - \eta \quad \text{and} \quad W(x_{k+1},x_i) \le a + \eta \quad \text{for all } i=2,\dots,k,
	\end{align}
	for $(x_1,\dots,x_k)$ and $(x_2,\dots,x_{k+1})$ equation \eqref{eq:graphon-matrix2} holds,
	\begin{equation}
	\begin{aligned}
	\label{eq:genbound2}
	W(x_1,y), W(x_{k+1},y) \in [a-\eta,b+\eta] \text{ for almost all } y\in [0,1],
	\end{aligned}
	\end{equation}
	and $Q$ is non-zero, where
	\begin{equation*}
	\begin{aligned}
	Q = & \left( \prod_{2 \le i < j \le k} W(x_i,x_j) \right) \\ & \cdot  \int_{[0,1]^{t-k}} \left(  \prod_{1 \le i < j \le t-k} W(y_i,y_j) \right) \left(   \prod_{2 \le i \le k, j \in [t-k]} W(x_i,y_j) \right) \prod_{i \in [t-k]} \operatorname{d}y_i.
	\end{aligned}
	\end{equation*}
	We get from \eqref{eq:graphon-matrix2} that
	
	\begin{equation*}
	\begin{aligned}
		& \left( \prod_{1 \le i < j \le k} W(x_i,x_j) \right) \\ & \cdot  \int_{[0,1]^{t-k}} \left(  \prod_{1 \le i < j \le t-k} W(y_i,y_j) \right) \left(   \prod_{1\le i\le k, j \in [t-k]} W(x_i,y_j) \right) \prod_{i \in [t-k]} \operatorname{d}y_i = p^{e(K_t)} \\
		= & \left( \prod_{2 \le i < j \le k+1} W(x_i,x_j) \right) \\ & \qquad \cdot  \int_{[0,1]^{t-k}} \left(  \prod_{1 \le i < j \le t-k} W(y_i,y_j) \right) \left(   \prod_{2 \le i \le k+1 , j \in [t-k]} W(x_i,y_j) \right) \prod_{i \in [t-k]} \operatorname{d}y_i.
	\end{aligned}
	\end{equation*}
	Using \eqref{eq:supbound2} and \eqref{eq:genbound2} we can lower bound the left hand side by
	\begin{align}
	\label{eq:lower}
	(b-\eta)^{k-1} (a-\eta)^{t-k} Q
	\end{align}
	and similarly we can upper bound the right hand side by 
	\begin{align}
	\label{eq:upper}
	(a+\eta)^{k-1} (b+\eta)^{t-k} Q.
	\end{align}
	As $Q$ is non-zero
	this gives
	\begin{align*}
	(b-\eta)^{k-1} (a-\eta)^{t-k} \le (a+\eta)^{k-1} (b+\eta)^{t-k}.
	\end{align*}
	If $t$ is even, we have $k-1>t-k$ and we can finish the proof similar to the case $t=4$.
	If $t$ is odd, i.e.\ $k-1=t-k$, a more refined argument is needed.
	
	Since $W$ is a graphon, the difference of~\eqref{eq:lower} and~\eqref{eq:upper} is at most $2t \eta$.
	In particular, the estimates used to derive~\eqref{eq:lower} cannot be too wasteful and it follows that
	\begin{align*}
	W(x_1,y) \le a+ 2 t \eta \quad \text{and} \quad W(x_{k+1},y) \ge b - 2 t \eta \quad \text{for almost all } y\in [0,1].
	\end{align*}
	Together with~\eqref{eq:genbound2} we get
	\begin{align*}
	\label{eq:Wcont1}
	W(x_1,y) \in [a \pm 2t \eta] \quad \text{and}  \quad W(x_{k+1},y) \in [b \pm 2t \eta] \quad \text{for almost all } y\in [0,1].
	\end{align*}
	It follows that $b-a\le 2(t+1)\eta$ and consequently $W(x,y)\in[p\pm 4(t+1)\eta]$ for almost all $(x,y)\in [0,1]^2$.
\end{proof}


\begin{thebibliography}{10}
\providecommand{\url}[1]{\texttt{#1}}
\providecommand{\urlprefix}{URL }
\providecommand{\eprint}[2][]{\url{#2}}

\bibitem{CGW89}
F.~R.~K. Chung, R.~L. Graham and R.~M. Wilson: \emph{Quasi-random graphs},
  Combinatorica \textbf{9} (1989), 345--362.

\bibitem{CFS}
D.~Conlon, J.~Fox and B.~Sudakov: \emph{An approximate version of {S}idorenko's
  conjecture}, Geom. Funct. Anal. \textbf{20} (2010), 1354--1366.

\bibitem{conlon2012weak}
D.~Conlon, H.~H{\`a}n, Y.~Person and M.~Schacht: \emph{Weak quasi-randomness
  for uniform hypergraphs}, Random Structures \& Algorithms \textbf{40} (2012),
  1--38.

\bibitem{han2011note}
H.~H{\`a}n, Y.~Person and M.~Schacht: \emph{Note on forcing pairs}, Electronic
  Notes in Discrete Mathematics \textbf{38} (2011), 437--442.

\bibitem{HH09}
H.~Hatami: \emph{Graph norms and {S}idorenko's conjecture}, Israel J. Math.
  (2010), 125--150.

\bibitem{KSSSz02}
J.~Koml{\'o}s, A.~Shokoufandeh, M.~Simonovits and E.~Szemer{\'e}di: \emph{The
  regularity lemma and its applications in graph theory}, in: Theoretical
  aspects of computer science ({T}ehran, 2000), \emph{Lecture Notes in Comput.
  Sci.}, volume 2292 (2002), 84--112.

\bibitem{KS06}
M.~Krivelevich and B.~Sudakov: \emph{Pseudo-random graphs}, in: More sets,
  graphs and numbers, \emph{Bolyai Soc. Math. Stud.}, volume~15 (2006),
  199--262.

\bibitem{lovasz2012large}
L.~Lov{\'a}sz: Large networks and graph limits, volume~60, 2012.

\bibitem{reiher2017forcing}
C.~Reiher and M.~Schacht: \emph{Forcing quasirandomness with triangles}, Forum
  of Mathematics, to appear in Forum of Mathematics, Sigma.

\bibitem{MS_comm}
M.~Schacht: Personal communication.

\bibitem{Sid93}
A.~Sidorenko: \emph{A correlation inequality for bipartite graphs}, Graphs
  Combin. \textbf{9} (1993), 201--204.

\bibitem{Sim84}
M.~Simonovits: \emph{Extremal graph problems, degenerate extremal problems, and
  supersaturated graphs}, in: Progress in graph theory ({W}aterloo, {O}nt.,
  1982) (1984), 419--437.

\bibitem{ST04}
J.~Skokan and L.~Thoma: \emph{Bipartite subgraphs and quasi-randomness}, Graphs
  Combin. \textbf{20} (2004), 255--262.

\bibitem{Thomason-2}
A.~Thomason: \emph{Pseudorandom graphs}, Random graphs '85 ({P}ozna\'n, 1985)
  (1987), 307--331.

\bibitem{Thomason-1}
A.~Thomason: \emph{Random graphs, strongly regular graphs and pseudorandom
  graphs}, in: Surveys in combinatorics 1987 ({N}ew {C}ross, 1987),
  \emph{London Math. Soc. Lecture Note Ser.}, volume 123 (1987), 173--195.

\end{thebibliography}
\end{document}